\documentclass[11pt,a4paper,reqno]{amsart}

\usepackage{amsmath,amssymb,amsthm,amsbsy,amsfonts}
\usepackage{enumerate}

\usepackage[bitstream-charter]{mathdesign}

\usepackage{graphicx}
\usepackage{esint}

\usepackage{amsthm}
\newtheorem{thm}{Theorem}
\newtheorem{prop}{Proposition}
\newtheorem{lem}{Lemma}

\theoremstyle{definition}
\newtheorem{defn}{Definition}
\newtheorem{ex}{Example}
\theoremstyle{remark}
\newtheorem{rem}{Remark}
\newtheorem*{nota}{Notation}

\newcommand{\bN}{\mathbb{N}}
\newcommand{\bR}{\mathbb{R}}

\newcommand{\cC}{\mathcal{C}}
\newcommand{\cE}{\mathcal{E}}
\newcommand{\cH}{\mathcal{H}}
\newcommand{\cL}{\mathcal{L}}

\newcommand{\cS}{\mathcal{S}}

\newcommand{\set}[1]{\left\{#1\right\}}
\newcommand{\norm}[1]{{\left\Vert #1 \right\Vert}}
\newcommand{\inp}[1]{{\left\langle #1 \right\rangle}}

\newcommand{\bv}{\mathbf{v}}

\DeclareMathOperator*{\diam}{diam}
\DeclareMathOperator*{\Hdim}{dim_H}
\DeclareMathOperator*{\Mdim}{dim_M}

\DeclareMathOperator*{\divv}{div}

\begin{document}
\baselineskip=18pt

\title[Singular points in NSE]{The Minkowski dimension of boundary singular points in the Navier--Stokes equations}

\author{Hi Jun Choe \& Minsuk Yang}

\address{Department of Mathematics, Yonsei University, 50 Yonsei-ro Seodaemun-gu, Seoul, Republic of Korea}

\email{choe@yonsei.ac.kr, m.yang@yonsei.ac.kr}

\begin{abstract}
We study the partial regularity problem of the three-dimensional incompressible Navier--Stokes equations.
We present a new boundary regularity criterion for boundary suitable weak solutions.
As an application, a bound for the parabolic Minkowski dimension of possible singular points on the boundary is obtained. \\

\end{abstract}

\maketitle

\section{Introduction}
\label{S1}

In this paper, we study the Minkowski dimension of the possible boundary singular points of boundary suitable weak solutions to the three-dimensional incompressible Navier--Stokes equations
\begin{align*}
\partial_t \bv + (\bv \cdot \nabla) \bv - \Delta \bv + \nabla \pi &= 0 \\
\divv \bv &= 0
\end{align*}
where $\bv$ is the velocity field of the fluid and $\pi$ is the scalar pressure.
The global wellposedness of the three-dimensional incompressible Navier--Stokes equations is one of the most important open questions in PDEs.
The study of wellposedness problem for the Navier--Stokes equations has long history and huge literature.
Long ago, Leray \cite{Leray} answered the existence of weak solutions.  But we do not know the global regularity and uniqueness in general.
There are several sufficient conditions which imply the regularity and uniqueness.
The Ladyzhenskaya--Prodi--Serrin condition is one of the most famous conditions.
Roughly speaking, if the velocity field possesses higher integrability, then the weak solution is regular and unique.
However, we could not bridge the gap between the higher integrability condition and the known existence condition.

In the 1970s, Scheffer \cite{Sch} introduced the concept of suitable weak solutions to the Navier--Stokes equations and presented a local regularity criterion.
We say that a space-time point $z=(x,t)$ is singular if the velocity field is not continuous at the point.
Then an interesting question would be investigating possible singular points.
The contrapositive of each regularity criterion yields some information about the possible singularities.
Surprisingly, from the structure of the equations, too many singular points can not exist.
There are several concepts reflecting the geometric size or distribution of sparse sets.
They are the main topics in the field of fractal geometry.
We refer the reader to Falconer's book \cite{F} for a brief introduction of the fractal geometry.
The two of the most fundamental tools are the Hausdorff dimension and the Minkowski (box-counting) dimension.
The Hausdorff measure is a natural generalization of the Lebesgue measure and the Hausdorff dimension reflect the geometric size of sets.
Actually, using Scheffer's idea, one can deduce that the parabolic version of the Hausdorff dimension of the possible singular points of suitable weak solutions for the Cauchy problem is bounded by $5/3$.
Thus, the possible singular points turned out to be very sparse.
After that, Caffarelli, Kohn, and Nirenberg \cite{CKN} devised the regularity criterion that there is an absolute positive number $\varepsilon$ such that the velocity field is locally bounded at $z$ if 
\[
\limsup_{r \to 0} r^{-1} \int_{Q(z,r)} |\nabla \bv|^2 dyds < \varepsilon
\]
where $Q(z,r)$ is parabolic cylinders and $dyds$ is the Lebesgue measures.
This criterion can give better information about the singular points.
The contrapositive of this criterion yields that the one dimensional parabolic Hausdorff measure of the singular points is zero.
Lin \cite{Lin} gave a new short proof by an indirect argument.
Ladyzhenskaya and Seregin \cite{LS} gave a clear presentation that the same condition implies that the velocity field is H\"older continuous at the point.
Thus, the distribution of singular points is extremely rare in the space-time domain so that one can not continuously observe any singular point for any short period.
Furthermore, Choe and Lewis \cite{CL} lowered an upper estimate of the parabolic Hausdorff measure of the singular points by a logarithmic factor.

The Minkowski dimension $\Mdim (S)$ of a set $S$ is closely related to the Hausdorff dimension $\Hdim (S)$, but the Minkowski dimension reflects the geometric complexity of the distribution of the sets.
Moreover, a good control of the Minkowski dimension has a stronger implication by the well-known relation
\[
\Hdim (S) \le \Mdim (S).
\]
This inequality can be strict in many instances.
However, for many self-similar sets, the Minkowski dimension and the Hausdorff dimension are the same due to the relatively simple geometric structure of those sets. 
In the next section we shall give simple examples to illustrate the difference of the two concepts.
There are some studies on estimating the Minkowski dimension of possible (interior) singular points for the Cauchy problem or the bounded domain case.
Although Scheffer did not mention the Minkowski dimension of the singular set, from his result, one can conclude that the Minkowski dimension is bounded by $5/3$.
The regularity criterion proved by Caffarelli, Kohn, and Nirenberg can not be used to estimate the Minkowski dimension of the singular points.
Robinson and Sadowski \cite{RS} presented other general conditions related to the Minkowski dimension of the interior singular points.
Recently, there are some efforts in order to lower the bound for the Minkowski dimension.
For example, the bound $135/82$ was given in Kukavica \cite{Ku}, $45/29$ in Kukavica and Pei \cite{KP}, $95/63$ in Koh and Yang \cite{KY}, and $360/277$ in Wang and Wu \cite{WW}.
They are still quite larger than the bound for the Hausdorff dimension.

In mathematical fluid mechanics, the boundary behavior of flows are most difficult. 
Moreover, it is very important to know the analysis of boundary layer, boundary singularities, vortex analysis, and related physical parameters including the Reynold number. 
The main objective of this paper is to investigate the boundary possible singular points.
Recently, Seregin, Shilkin, and Solonnikov \cite{Se1, Se2, Se3, SS, SSS} studied the partial boundary regularity and extended some of the fundamental regularity criteria to the boundary cases.
Although the boundary is a lower dimensional manifold, the boundary regularity criteria only yields that the Hausdorff dimension of the boundary singular points is bounded by 1, which is the same bound for the interior singular points.
In this paper we investigate the Minkowski dimension of the boundary singular points.
Here is our main result.

\begin{thm}
\label{T}
The parabolic Minkowski dimension of any compact subset of the boundary singular points is bounded by $3/2$.
\end{thm}

We end this section by mentioning a few remarks.
To the best of the authors knowledge, this theorem is the first result about the Minkowski dimension of the boundary singular points.
To prove the theorem, we present a special boundary regularity criterion that is Proposition \ref{P} in the last section.
We obtained the bound basically adapting the strategy developed in \cite{KY} with an idea in \cite{WW}.
But, the technical details are quite different from the interior analysis.
For the boundary analysis, we carefully chose several indices related to the integrability exponents of the pressure and the shrinking ratio of decaying estimates  of the scaled functionals.

\section{The Minkowski Dimension}
\label{S2}

In this section we give formal definitions of the Hausdorff dimension and the Minkowski dimension in a general metric space setting and then give two simple examples to illustrate the difference between the two concepts.
This short section can be safely skipped for the reader who is familiar with those concepts.

\begin{defn}[The Hausdorff dimension]
Given a set $S$ in a metric space $(X,d)$ and $\delta>0$, we denote by $\cC(S,\delta)$ the family of all countable coverings $\set{E_k}$ that covers $S$ with $\diam(E_k) \le \delta$, where the diameter of a set $E$ is defined by 
\[
\diam(E) := \sup\set{d(x,y) : x, y \in E}.
\]
Then the $\alpha$ dimensional Hausdorff measure is defined by 
\[
\cH^\alpha(S) = \lim_{\delta\to0} \inf\set{\sum_k \diam(E_k)^\alpha : \set{E_k} \in \cC(S,\delta)}
\]
and the Hausdorff dimension of the set $S$ is defined by
\[
\Hdim(S) = \inf\set{\alpha : \cH^\alpha(S)=0}.
\]
\end{defn}

\begin{defn}[The Minkowski dimension]
Given a set $S$ in a metric space $(X,d)$ and $\delta>0$, we denote by $N(S,\delta)$ the minimum number of all finite coverings $\set{E_k}$ that covers $S$ with $\diam(E_k) = \delta$.
Then the Minkowski dimension of the set $S$ is defined as
\[
\Mdim(S) = \lim_{\delta\to0} \frac{\log N(S,\delta)}{-\log \delta}.
\]
\end{defn}

The Minkowski dimension is closely related to the complexity of the geometric distribution.
For the most self-similar sets, the Hausdorff dimension and the Minkowski dimension are the same since the self-similar sets have some symmetries involving scaling and translation.
Here is the first example of this section.

\begin{ex}
For any $0 < \alpha < 1$ there is a compact set $C \subset [0,1]$ such that 
\[
\Hdim (C)=\Mdim (C)=\alpha.
\]
\end{ex}

We first construct $C$ as a Cantor-type set by inductively removing varying portions from the middle of each interval. 
Let $C_0 = I_0^1 = [0,1]$ and set 
\[
\delta_k = 2^{-k/\alpha} \quad \text{ for } k = 0,1,2,\cdots.
\]
Suppose the set $C_{k-1}$ has been constructed and satisfies 
\[
C_{k-1} = \bigcup_{j=1}^{2^{k-1}} I_{k-1}^j, \quad \diam(I_{k-1}^j) = \delta_{k-1}.
\]
We divide each interval $I_{k-1}^j$ into two closed intervals $I_k^{2j-1}$ and $I_k^{2j}$ by removing an open interval of length $\delta_{k-1}-2\delta_k$ from the middle of $I_{k-1}^j$ so that 
\[
C_k = \bigcup_{j=1}^{2^k} I_k^j, \quad \diam(I_k^j) = \delta_k.
\]
By continuing this process, we obtain the Cantor-type set
\[
C = \bigcap_{k=0}^\infty C_k.
\]

We now show that $\Hdim (C)=\Mdim (C)=\alpha$.
Clearly, the closed intervals $I_k^j$ consisting of $C_{k}$ cover the set $C$ so that 
\begin{equation}
\label{E24}
\cH^\alpha(C) 
= \lim_{k \to \infty} \sum_{j=1}^{2^k} \diam(I_k^j)^{\alpha} 
= \lim_{k \to \infty} 2^k \delta_k^\alpha = 1
\end{equation}
where $\cH^\alpha$ denote the $\alpha$-dimensional Hausdorff measure.
Thus, $\Hdim (C)=\alpha$.
On the other hand, we have 
\begin{equation}
\label{E25}
\Mdim (C) = \lim_{k \to \infty} \frac{\log 2^k}{- \log \delta_k} = \lim_{k \to \infty} \frac{\log 2^k}{- \log 2^{-k/\alpha}} = \alpha.
\end{equation}
In fact, some extra efforts are needed to check the validity of the first equalities of \eqref{E24} and \eqref{E25}, but we omit the details.

The second example shows that the Minkowski dimension can be much larger than the Hausdorff dimension because simple condensation break symmetry and increase complexity. 

\begin{ex}
For any $0 < \alpha < 1$ there is a compact set $S \subset [0,1]^2$ such that 
\begin{align}
\label{E21}
\Hdim (S) &= \alpha \\
\label{E22}
\Mdim (S) &= \alpha + \frac{1}{2}.
\end{align}
\end{ex}

Let $J = \set{0} \cup \set{n^{-1} : n \in \bN}$ and set
\[
S = C \times J \subset [0,1]^2
\]
where $C$ is the set in Example 1.
Then $\cH^\alpha(C \times \set{0}) > 0$ and for any positive number $\epsilon$
\[
\cH^{\alpha+\epsilon}(C \times \set{k^{-1}}) = 0
\]
and hence $\cH^{\alpha+\epsilon}(S)=0$ by the countable sub-additivity and the translation invariant property of Hausdorff measures.
This shows \eqref{E21}.

On the other hand, we fix $\frac{1}{k(k+1)} \le \delta < \frac{1}{k(k-1)}$, then $N(S,\delta) = N(C,\delta) N(J,\delta)$, so 
\[
\frac{\log N(S,\delta)}{-\log \delta} = \frac{\log N(C,\delta)+\log N(J,\delta)}{-\log \delta}.
\]
Since $k < N(J,\delta) < 2k$, we have 
\[
\Mdim (S) = \Mdim (C) + \lim_{k\to\infty} \frac{\log N(J,\delta)}{-\log k^{-2}} = \alpha + \frac{1}{2}.
\]
All the computations can be justified under the control of negligible errors.
By modifying the set $J$, one can construct a compact set $S \subset [0,1]^2$ satisfying $\dim_M(S) = \alpha + \beta$ for any $0 < \beta < 1$.

\section{Preliminaries}
\label{S3}

In this section we give the definitions of the parabolic Hausdorff dimension, the parabolic Minkowski dimension, and suitable weak solutions.
We also set up our shorthand-notations for complicated scaled functionals, which will be helpful to figure out clearly the iteration process in the proof.
We end this section by giving a simple lemma which is an immediate consequence of the fundamental regularity criterion in Seregin \cite{Se2}.

We denote space balls centered at $x$ and parabolic cylinders centered at a space-time point $z=(x,t)$ by
\begin{align*}
B(x,r) &= \set{y\in\bR^3:|y-x|<r}, \\
Q(z,r) &= B(x,r) \times (t-r^2,t).
\end{align*}

\begin{defn}[The parabolic Hausdorff dimension]
Given a set $S\subset\bR^3\times\bR$ and a positive number $\delta$, we denote by $\cC_p(S,\delta)$ the collection of all coverings of parabolic cylinders $\set{Q(z_k,r_k)}$ that covers the set $S$ with $0<r_k \le \delta$.
Then the $\alpha$ dimensional parabolic Hausdorff measure is defined by
\[
\cH_p^\alpha(S) = \lim_{\delta\to0} \inf \set{\sum_k r_k^\alpha : \set{Q(z_k,r_k)} \in \cC_p(S,\delta)}.
\]
The parabolic Hausdorff dimension of the set $S$ is defined by
\[
\Hdim(S) = \inf\set{\alpha : \cH_p^\alpha(S)=0}.
\]
\end{defn}

\begin{defn}[The parabolic Minkowski dimension]
Given a set $S\subset\bR^3\times\bR$ and a positive number $\delta$, we denote by $N(S,\delta)$ the minimum number of parabolic cylinders $\set{Q(z,\delta)}$ required to cover the set $S$.
Then the parabolic (upper) Minkowski dimension of the set $S$ is defined as
\begin{equation}
\label{E30}
\Mdim(S) = \limsup_{\delta\to0} \frac{\log N(S,\delta)}{-\log \delta}.
\end{equation}
\end{defn}

\begin{rem}
We use parabolic cylinders instead of arbitrary sets whose diameters are restricted in terms of the parabolic distance $d_p((x,t),(y,s)) = |x-y|+\sqrt{|t-s|}$.
\end{rem}

Now, we recall the definition of suitable weak solutions in Seregin and Shilkin \cite{SS}.
From the nature of the local regularity theory, we may consider the fixed domain $Q=B\times(-1,0)$ and $Q^+=B^+\times(-1,0)$ where $B = \set{x \in \bR^3 : |x|<1}$ and $B^+ = \set{x \in B : x_3>0}$.
We shall write $f \in \cL^{p,q}(Q)$ if
\begin{equation}
\norm{f}_{\cL^{p,q}(Q)} := \left(\int_{-1}^0\left(\int_B |f(x,t)|^p dx\right)^{q/p} dt\right)^{1/q} < \infty
\end{equation}
and simply put $L^p(Q) = \cL^{p,p}(Q)$.

\begin{defn}[Interior suitable weak solutions] 
A pair $(\bv,\pi)$ is called a (interior) suitable weak solution to the Navier--Stokes equations in $Q$ if the following three conditions are satisfied:
\begin{enumerate}
\item
$\bv \in \cL^{2,\infty}(Q)$, $\nabla\bv \in L^2(Q)$, and $\pi \in L^{3/2}(Q)$.
\item
$(\bv,\pi)$ satisfies the Navier--Stokes equations in $Q$ in the sense of distributions.
\item
$(\bv,\pi)$ satisfies the local energy inequality in $Q$
\begin{align*}
&\int_B |\bv(x,t)|^2 \phi(x,t) dx + 2 \int_{-1}^t \int_B |\nabla \bv|^2 \phi dx dt \\
&\le \int_{-1}^t \int_B |\bv|^2 (\partial_t \phi + \Delta \phi) + (|\bv|^2 + 2\pi) \bv \cdot \nabla \phi dx dt
\end{align*}
for almost all $t \in (-1,0)$ and for any non-negative $\phi \in C^\infty(\bR^3\times\bR)$ vanishing near the parabolic boundary $\partial B \times (-1,0) \cup B \times \set{t=-1}$.
\end{enumerate}
\end{defn}

\begin{defn}[Boundary suitable weak solutions]  
A pair $(\bv,\pi)$ is called a boundary suitable weak solution to the Navier--Stokes equations in $Q^+$ if the following four conditions are satisfied:
\begin{enumerate}
\item
$\bv \in \cL^{2,\infty}(Q^+)$, $\nabla\bv \in L^{2}(Q^+)$, and $\pi \in L^{3/2}(Q^+)$.
\item
$\bv|_{x_3=0}=0$ in the sense of traces.
\item
$(\bv,\pi)$ satisfies the Navier--Stokes equations in $Q^+$ in the sense of distributions.
\item
$(\bv,\pi)$ satisfies the local energy inequality in $Q^+$
\begin{equation}
\begin{split}
\label{E31}
&\int_{B^+} |\bv(x,t)|^2 \phi(x,t) dx + 2 \int_{-1}^t \int_{B^+} |\nabla \bv|^2 \phi dx dt \\
&\le \int_{-1}^t \int_{B^+} |\bv|^2 (\partial_t \phi + \Delta \phi) dxdt 
+ \int_{-1}^t \int_{B^+} |\bv|^2 \bv \cdot \nabla \phi dx dt
+ 2 \int_{-1}^t \int_{B^+} \pi \bv \cdot \nabla \phi dx dt
\end{split}
\end{equation}
for almost all $t \in (-1,0)$ and for any non-negative $\phi \in C^\infty(\bR^3\times\bR)$ vanishing near the parabolic boundary $\partial B \times (-1,0) \cup B \times \set{t=-1}$.
\end{enumerate}
\end{defn}

In the papers \cite{Se1, Se2} and \cite{SSS}, the definition of boundary suitable weak solutions is different.
It is supposed in addition that the second spatial derivatives and the first derivative in time of the velocity field and the gradient of the pressure exist as integrable functions in $Q^+$.
We adopt the definition in \cite{SS} and there it was pointed that extra regularity assumptions are simply superfluous.
We refer the reader to \cite{SS} for detailed explanation of the concept of solutions.
For notational convenience we shall use the following shorthand notations.

\begin{nota}
We denote $A \lesssim B$ if there exists a generic positive constant $C$ such that $|A| \le C|B|$. 
We denote the average value of $f$ on $E$ by 
\[
\inp{f}_{E} = \fint_E f d\mu = |E|^{-1} \int_E f dx
\]
where $|E|$ represents the three dimensional Lebesgue measure of the set $E$ in $\bR^3$.
\end{nota}

\begin{defn}[Scaled functionals]
Let $z=(x,t) \in \Gamma_T$ and define 
\begin{align*}
A(z,r) &= r^{-1} \sup_{|t-s|<r^2} \int_{B^+(x,r)} |\bv(y,s)|^2 dy \\
E(z,r) &= r^{-1} \int_{Q^+(z,r)} |\nabla \bv(y,s)|^2 dyds \\
F(z,r) &= r^{-4/3} \norm{\bv}_{L^{3}(Q^+(z,r))}^2 \\
G(z,r) &= r^{-1} \norm{\nabla \pi}_{\cL^{9/8,3/2}(Q^+(z,r))} \\
Y(z,r) &= F(z,r) + G(z,r).
\end{align*}
\end{defn}

\begin{rem}
We suppress the parameter $z$ when it is a fixed reference point and it can be understood obviously in the context.
\end{rem}

We end this section by giving the following lemma which is a direct consequence of fundamental regularity criterion in Seregin \cite{Se2}.

\begin{lem}
\label{L1}
There exists a positive constant $\varepsilon$ such that $\bv$ is regular at a boundary point $z \in \Gamma_T$ if $Y(z,R) < \varepsilon$ for some positive number $R$.
\end{lem}

\begin{proof}
From Seregin \cite{Se2} it is well-known that there exists a positive constant $\eta$ such that $\bv$ is regular at a boundary point $z \in \Gamma_T$ if for some positive number $R$
\[
R^{-2} \int_{Q^+(z,R)} |\bv(y,s)|^3 + |\pi(y,s)-\inp{\pi}_{B^+(x,R)}|^{3/2} dyds < \eta.
\]
Clearly, we have 
\[
R^{-2} \int_{Q^+(z,R)} |\bv|^3 = F(z,R)^{3/2}.
\]
By using the Young inequality and the Sobolev inequality, we also have 
\begin{align*}
&R^{-2} \int_{t-R^2}^t \int_{B^+(x,R)} |\pi-\inp{\pi}_{B^+(x,R)}|^{3/2} dyds \\
&\le R^{-2} \int_{t-R^2}^t \left(\int_{B^+(x,R)} 1 dy\right)^{1/6} \left(\int_{B^+(x,R)} |\pi-\inp{\pi}_{B^+(x,R)}|^{9/5} dy\right)^{5/6} ds \\
&\le C R^{-3/2} \int_{t-R^2}^t \left(\int_{B^+(x,R)} |\nabla \pi|^{9/8} dy\right)^{4/3} ds \\
&= C G(z,R)^{3/2}.
\end{align*}
Thus, we can take $\varepsilon = (1+C)^{-2/3} \eta^{2/3}$ so that 
\[
R^{-2} \int_{Q^+(z,R)} |\bv|^3 + |\pi-\inp{\pi}_{B^+(x,R)}|^{3/2} dyds \le (1+C) Y(z,R)^{3/2} 
< (1+C) \varepsilon^{3/2} = \eta.
\]
\end{proof}

\section{Auxiliary lemmas}
\label{S4}

In this section, we present a few inequalities among the scaled functionals, which are very important to complete iteration schemes and to obtain better bounds for the Minkowski dimension.
The first inequality is a direct consequence of the local energy inequality \eqref{E31}.

\begin{lem}
\label{L2}
There is a constant $K_1>1$ such that for any $z \in \Gamma_T$, $0<r<1$, 
\begin{align*}
&rA(z,r/2) + rE(z,r/2) \\
&\le K_1 r^{-1/2} \int_{Q^+(z,r)} |\bv|^{10/3} + |\nabla \bv|^2 + |\pi-\inp{\pi}_{B^+(x,r)}|^{5/3} + |\nabla\pi|^{5/4} dyds.
\end{align*}
\end{lem}

\begin{proof}
We shall use a smooth cutoff function $\phi$ supported in $Q(z,r)$ and fulfilled $\phi = 1$ in $Q(z,r/2)$ and 
\[|\partial_t\phi|+|\nabla^2\phi|+|\nabla\phi|^2 \le 100 r^{-2} \quad \text{ in } Q(z,r).\]
Since $\bv$ vanishes on the boundary, we use the Poincar\'e inequality in any time to estimate the first integral on the right of \eqref{E31} as
\[
\int_{Q^+(z,r)} |\bv|^2 (\partial_t \phi + \Delta \phi) 
\lesssim \int_{Q^+(z,r)} |\nabla \bv|^2 
\]
where we omit the symbol $dy ds$ representing the Lebesgue measures. 
Since $\bv$ vanishes on the boundary, we estimate the second integral on the right of \eqref{E31} as
\begin{align*}
&\int_{t-r^2}^t \int_{B^+(x,r)} |\bv|^2 \bv \cdot \nabla \phi \\
&\lesssim r^{-1} \int_{t-r^2}^t \left(\int_{B^+(x,r)} 1\right)^{1/6} \left(\int_{B^+(x,r)} |\bv|^{10/3}\right)^{3/4} \left(\int_{B^+(x,r)} |\bv|^{6}\right)^{1/12} \\
&\lesssim r^{-1/2} \int_{t-r^2}^t \left(\int_{B^+(x,r)} |\bv|^{10/3}\right)^{3/4} \left(\int_{B^+(x,r)} |\nabla \bv|^2\right)^{1/4} \\
&\lesssim r^{-1/2} \int_{t-r^2}^t \int_{B^+(x,r)} |\bv|^{10/3} + |\nabla \bv|^2
\end{align*}
by the H\"older inequality, the Sobolev inequality, and the Young inequality.
Since $\divv \bv=0$, we can subtract an average $\inp{\pi} := \inp{\pi}_{B^+(x,r)}$ from the integrand of the third integral of \eqref{E31} to get the last estimate in the lemma.
Indeed, we have
\begin{align*}
&\int_{t-r^2}^t \int_{B^+(x,r)} \pi \bv \cdot \nabla \phi \\
&\lesssim r^{-1} \int_{t-r^2}^t \int_{B^+(x,r)} |\bv| |\pi-\inp{\pi}| \\
&\lesssim r^{-1}  \int_{t-r^2}^t \left(\int_{B^+(x,r)} |\bv|^{10/3} \right)^{3/10} 
\left(\int_{B^+(x,r)} |\pi-\inp{\pi}|^{5/3} \right)^{3/10} 
\left(\int_{B^+(x,r)} |\pi-\inp{\pi}|^{5/4} \right)^{4/10} \\
&\lesssim r^{-1/2}  \int_{t-r^2}^t \left(\int_{B^+(x,r)} |\bv|^{10/3} \right)^{3/10} 
\left(\int_{B^+(x,r)} |\pi-\inp{\pi}|^{5/3} \right)^{3/10} 
\left(\int_{B^+(x,r)} |\nabla\pi|^{5/4} \right)^{4/10} \\
&\lesssim r^{-1/2} \int_{t-r^2}^t \int_{B^+(x,r)} |\bv|^{10/3} + |\pi-\inp{\pi}|^{5/3} + |\nabla\pi|^{5/4} 
\end{align*}
by the H\"older inequality, the Poincar\'e inequality, and the Young inequality.
We notice that all implied constants in this proof are absolute.
\end{proof}

The second inequality is the following interpolation inequality.

\begin{lem}
\label{L3}
There is a constant $K_2>1$ such that for any $z \in \Gamma_T$, $0<r \le 1$, and $0<\theta \le 1$, 
\[
F(z,\theta r) \le K_2 \theta^{-1} A(z,r)^{1/3} E(z,r)^{2/3}.
\]
\end{lem}

\begin{proof}
We may assume $x=0$ and $r=1$.
Using the Young inequality and the Sobolev inequality we get 
\begin{align*}
\int_{B^+(\theta)} |\bv|^3 
&\le \left(\int_{B^+(\theta)} 1 dy\right)^{1/6} 
\left(\int_{B^+(\theta)} |\bv|^2\right)^{1/2} 
\left(\int_{B^+(\theta)} |\bv|^6\right)^{1/3} \\
&\le \theta^{1/2}
\left(\int_{B^+(1)} |\bv|^2\right)^{1/2} 
\left(\int_{B^+(1)} |\bv|^6\right)^{1/3} \\
&\lesssim \theta^{1/2} A(z,1)^{1/2} 
\int_{B^+(1)} |\nabla \bv|^2.
\end{align*}
Integrating in time over $(t-\theta^2,t)$ we obtain
\begin{align*}
\left(\int_{t-\theta^2}^t \int_{B^+(\theta)} |\bv|^3 dy ds\right)^{2/3} 
&\lesssim \left(\theta^{1/2} A(z,1)^{1/2} 
\int_{t-1}^t \int_{B^+(1)} |\nabla \bv|^2\right)^{2/3} \\
&\le \theta^{1/3} A(z,1)^{1/3} E(z,1)^{2/3}.
\end{align*}
Multiplying by $\theta^{-4/3}$ yields the result.
We notice that all implied constants in the proof are absolute.
\end{proof}

The third inequality is the following decay estimate for the pressure, which is a modification of Lemma 11 in Gustafson, Kang, and Tsai \cite{GKT} (see also Seregin \cite{Se1} and Seregin, Shilkin, and Solonnikov \cite{SSS}).

\begin{lem} 
\label{L4}
There is a constant $K_3>1$ such that for any $z \in \Gamma_T$, $0 < r \le 1$, and $0<\theta \le 1/4$, 
\begin{align*}
G(z,\theta r) 
&\le K_3 \theta r^{-2} \norm{\pi-\inp{\pi}}_{\cL^{9/8,3/2}(Q^+(z,r))} \\
&\quad + K_3 \theta E(z,r)^{1/2} + K_3 \theta^{-1} A(z,r)^{1/3} E(z,r)^{2/3}.
\end{align*}
\end{lem}

\begin{proof}
We may assume $x=0$ and $r=1$.
We fix a smooth domain $\frac{1}{2}B^+ \subset \widetilde{B}^+ \subset B^+$ and denote $\widetilde{Q}^+ = \widetilde{B}^+ \times (t-1,t)$.
Let $(\bv_1,\pi_1)$ be the unique solution to the initial boundary value problem for the Stokes system
\begin{align*}
\partial_t \bv_1 - \Delta \bv_1 + \nabla \pi_1 &= -(\bv \cdot \nabla) \bv \\
\divv \bv_1 &= 0
\end{align*}
in $\widetilde{Q}^+$ with $|\widetilde{B}^+|^{-1} \int_{\widetilde{B}^+} \pi_1(y,s) dy = 0$ for all $s \in (t-1,t)$ and 
\[
\bv_1(y,s) = 0, \qquad (y,s) \in (\partial \widetilde{B}^+ \times [t-1,t]) \cup (\widetilde{B}^+ \times \set{t-1}).
\]
Due to Theorem 3.1 in Giga and Sohr \cite{GS} we have 
\begin{equation}
\label{E41}
\norm{\nabla \bv_1}_{\cL^{9/8,3/2}(\widetilde{Q}^+)} 
+ \norm{\pi_1}_{\cL^{9/8,3/2}(\widetilde{Q}^+)} 
+ \norm{\nabla \pi_1}_{\cL^{9/8,3/2}(\widetilde{Q}^+)} 
\lesssim \norm{(\bv \cdot \nabla) \bv}_{\cL^{9/8,3/2}(\widetilde{Q}^+)}. 
\end{equation}

Let 
\begin{equation}
\label{E42}
\bv_2 = \bv - \bv_1, \qquad \pi_2 = \pi - \inp{\pi} - \pi_1
\end{equation}
where $\inp{\pi} = |\frac{1}{2} B^+|^{-1} \int_{\frac{1}{2} B^+} \pi dy$.
Then $(\bv_2,\pi_2)$ satisfies 
\begin{align*}
\partial_t \bv_2 - \Delta \bv_2 + \nabla \pi_2 &= 0 \\
\divv \bv_2 &= 0
\end{align*}
in $\widetilde{Q}^+$ and 
\[
\bv_2(y,s) = 0, \qquad (y,s) \in (\partial \widetilde{B} \cap \set{x_3=0}) \times [t-1,t].
\]
Due to Proposition 2 in Seregin \cite{Se1}, we have 
\[
\norm{\nabla \pi_2}_{\cL^{9/2,3/2}(\frac{1}{4}Q^+)} 
\lesssim \norm{\pi_2}_{\cL^{9/8,3/2}(\frac{1}{2}Q^+)}
+ \norm{\bv_2}_{\cL^{9/8,3/2}(\frac{1}{2}Q^+)} 
+ \norm{\nabla \bv_2}_{\cL^{9/8,3/2}(\frac{1}{2}Q^+)}.
\]
Thus, we use \eqref{E42}, the Sobolev inequality, and \eqref{E41} to get 
\begin{equation}
\label{E43}
\begin{split}
\norm{\nabla \pi_2}_{\cL^{9/2,3/2}(\frac{1}{4}Q^+)} 
&\lesssim \norm{\pi - \inp{\pi}}_{\cL^{9/8,3/2}(\frac{1}{2}Q^+)} 
 + \norm{\pi_1}_{\cL^{9/8,3/2}(\frac{1}{2}Q^+)} \\
&\quad + \norm{\nabla \bv}_{\cL^{9/8,3/2}(\frac{1}{2}Q^+)} 
+ \norm{\nabla \bv_1}_{\cL^{9/8,3/2}(\frac{1}{2}Q^+)} \\
&\lesssim \norm{\pi - \inp{\pi}}_{\cL^{9/8,3/2}(\frac{1}{2}Q^+)} 
+ \norm{\nabla \bv}_{\cL^{9/8,3/2}(\frac{1}{2}Q^+)} 
+ \norm{(\bv \cdot \nabla) \bv}_{\cL^{9/8,3/2}(\widetilde{Q}^+)}
\end{split}
\end{equation}
We use the H\"older inequality and combine \eqref{E41} and \eqref{E43} to obtain that for $0 < \theta \le 1/4$ 
\begin{equation}
\label{E44}
\begin{split}
\norm{\nabla \pi}_{\cL^{9/8,3/2}(\theta Q^+)} 
&\le \norm{\nabla \pi_1}_{\cL^{9/8,3/2}(\theta Q^+)} 
+ \norm{\nabla \pi_2}_{\cL^{9/8,3/2}(\theta Q^+)} \\
&\lesssim \norm{\nabla \pi_1}_{\cL^{9/8,3/2}(\frac{1}{2}Q^+)} 
+ \theta^2 \norm{\nabla \pi_2}_{\cL^{9/2,3/2}(\theta Q^+)} \\
&\lesssim \norm{(\bv \cdot \nabla) \bv}_{\cL^{9/8,3/2}(\widetilde{Q}^+)}
+ \theta^2 \norm{\pi - \inp{\pi}}_{\cL^{9/8,3/2}(\frac{1}{2}Q^+)} 
+ \theta^2 \norm{\nabla \bv}_{\cL^{9/8,3/2}(\frac{1}{2}Q^+)}.
\end{split}
\end{equation}
Since $\bv$ vanishes on the boundary, we have by using the H\"older's inequality and the Sobolev inequality
\[
\norm{(\bv \cdot \nabla) \bv}_{9/8}
\le \norm{\bv}_2^{2/3} \norm{\bv}_6^{1/3} \norm{\nabla \bv}_2 
\lesssim \norm{\bv}_2^{2/3} \norm{\nabla \bv}_2^{4/3}. 
\] 
Integrating in time yields
\[
\norm{(\bv \cdot \nabla) \bv}_{\cL^{9/8,3/2}(\widetilde{Q}^+)}
\lesssim A(z,1)^{1/3} E(z,1)^{2/3}.
\]
By H\"older's inequality 
\[
\norm{\nabla \bv}_{\cL^{9/8,3/2}(\frac{1}{2}Q^+)} \lesssim \norm{\nabla \bv}_{\cL^{2,2}(\frac{1}{2}Q^+)} \lesssim E(z,1)^{1/2}.
\]
From \eqref{E44} we obtain
\[
\theta G(\theta) 
\le \theta^2 \norm{\pi - \inp{\pi}}_{\cL^{9/8,3/2}(\frac{1}{2}Q^+)} + \theta^2 E(z,1)^{1/2} + 
A(z,1)^{1/3} E(z,1)^{2/3}.
\]
This yields the result.
\end{proof}

\section{Proof of Theorem \ref{T}}
\label{S5}

In this section we prove Proposition \ref{P} and then deduce Theorem \ref{T} from it. 
When one investigate the Minkowski dimension of the singular points, a plausible strategy is combining the different scaled functionals to lower the power of $\rho$ in the right-hand side of \eqref{E51}. 
We note that Wang and Wu \cite{WW} observed that adding the term $|\nabla \pi|^{5/4}$ in \eqref{E51} is useful to get better bound for the Minkowski dimension of the interior singular points compared with the original argument in Koh and Yang \cite{KY}.
We adopt the same term in this boundary criterion and revise technical details of the iteration scheme due to the different decaying behavior of the scaled functional of the pressure near the boundary.

\begin{prop}
\label{P}
There exists a positive number $\cE<1$ such that the point $z \in \Gamma_T$ is regular if for some positive number $\rho < 2^{-12}$
\begin{equation}
\label{E51}
\int_{Q^+(z,\rho)} |\bv|^{10/3} + |\nabla \bv|^2 + |\pi-\inp{\pi}_{B^+(x,\rho)}|^{5/3} + |\nabla \pi|^{5/4} dxdt < \rho^{3/2} \cE.
\end{equation}
\end{prop}

We divide the proof of Proposition \ref{P} into a few steps.
We suppress $z$ as a matter of convenience. 

\begin{proof}
\begin{enumerate}[\bf{Step} 1)]
\item
Suppose that for some fixed positive number $\rho < 2^{-12}$
\begin{equation}
\label{E52}
\int_{Q^+(2\rho)} |\bv|^{10/3} + |\nabla \bv|^2 + |\pi-\inp{\pi}_{B^+(x,2\rho)}|^{5/3} + |\nabla \pi|^{5/4} dxdt < (2\rho)^{3/2} \cE.
\end{equation}
Then, from the definition $E(\rho) = \rho^{-1} \int_{Q^+(\rho)} |\nabla \bv|^2 dxdt$, we have 
\begin{equation}
\label{E53}
E(\rho) < 4 \rho^{1/2} \cE.
\end{equation}
By Lemma \ref{L2} and \eqref{E52} we also have
\[
A(\rho) 
\le K_1 (2\rho)^{-3/2} \int_{Q^+(2\rho)} |\bv|^{10/3} + |\nabla \bv|^2 + |\pi-\inp{\pi}_{B^+(x,2\rho)}|^{5/3} + |\nabla\pi|^{5/4}
\]
and hence 
\begin{equation}
\label{E54}
A(\rho) < K_1 \cE.
\end{equation}
\item
Let 
\begin{equation}
\label{E55}
\alpha := \frac{7}{6}, \qquad \beta := \frac{1}{6}, \qquad \theta := \rho^\beta < (2^{-12})^{1/6} = \frac{1}{4}.
\end{equation}
Combining Lemma \ref{L3} and Lemma \ref{L4} we obtain that 
\begin{equation}
\label{E56}
\begin{split}
Y(\theta \rho^{\alpha}) 
&\le K_3 \theta \rho^{-2\alpha} \norm{\pi-\inp{\pi}}_{\cL^{9/8,3/2}(Q^+(\rho^{\alpha}))} \\
&\quad + K_3 \theta E(\rho^{\alpha})^{1/2} + (K_2+K_3) \theta^{-1} A(\rho^{\alpha})^{1/3} E(\rho^{\alpha})^{2/3}.
\end{split}
\end{equation}
We now estimate the first term on the right as follows.
Using the H\"older inequality, the Sobolev inequality, and the Young inequality, we obtain 
\begin{align*}
&\left(\int_{B^+(\rho^{\alpha})} |\pi-\inp{\pi}|^{9/8}\right)^{4/3} \\
&\le 
\left(\int_{B^+(\rho^{\alpha})} 1\right)^{1/2}
\left(\int_{B^+(\rho^{\alpha})} |\pi-\inp{\pi}|^{5/3}\right)^{3/5}
\left(\int_{B^+(\rho^{\alpha})} |\pi-\inp{\pi}|^{15/7}\right)^{7/30} \\
&\le C \rho^{3\alpha/2}
\left(\int_{B^+(\rho^{\alpha})} |\pi-\inp{\pi}|^{5/3}\right)^{3/5}
\left(\int_{B^+(\rho^{\alpha})} |\nabla\pi|^{5/4}\right)^{2/5} \\
&\le C \rho^{3\alpha/2} \int_{B^+(\rho^{\alpha})} |\pi-\inp{\pi}|^{5/3} + |\nabla\pi|^{5/4}
\end{align*}
and hence 
\begin{align*}
&\rho^{-\alpha} \norm{\pi-\inp{\pi}}_{\cL^{9/8,3/2}(Q^+(\rho^{\alpha}))} \\
&= \rho^{-\alpha} \left(\int_{t-\rho^{2\alpha}}^t \left(\int_{B^+(\rho^{\alpha})} |\pi-\inp{\pi}|^{9/8}\right)^{4/3}\right)^{2/3} \\
&\le C \left(\int_{Q^+(\rho^\alpha)} |\pi-\inp{\pi}|^{5/3} + |\nabla\pi|^{5/4}\right)^{2/3}.
\end{align*}
Thus, from \eqref{E56}, we have  
\begin{equation}
\label{E57}
\begin{split}
Y(\theta \rho^{\alpha}) 
&\le C K_3 \theta \rho^{-\alpha} \left(\int_{Q^+(\rho^\alpha)} |\pi-\inp{\pi}|^{5/3} + |\nabla\pi|^{5/4}\right)^{2/3} \\
&\quad + K_3 \theta E(\rho^{\alpha})^{1/2} + (K_2+K_3) \theta^{-1} A(\rho^{\alpha})^{1/3} E(\rho^{\alpha})^{2/3} \\
&=: I + II + III.
\end{split}
\end{equation}
\item
We shall estimate $I$, $II$, and $III$.
Using \eqref{E52} and \eqref{E55} we have
\begin{equation}
\label{E58}
I 
\le C K_3 \rho^\beta \rho^{-\alpha} ((2\rho)^{3/2} \cE)^{2/3} \\
\le 4C K_3 \cE^{2/3}.
\end{equation}
Using \eqref{E53} and \eqref{E55} we have
\begin{equation}
\label{E59}
\begin{split}
II
&= K_3 \theta E(\rho^\alpha)^{1/2} \\
&\le K_3 \theta \Big(\frac{\rho}{\rho^\alpha}\Big)^{1/2} E(\rho)^{1/2} \\
&\le K_3 \rho^\beta \rho^{(1-\alpha)/2} (4 \rho^{1/2} \cE)^{1/2} \\
&\le 2 K_3 \cE^{1/2}.
\end{split}
\end{equation}
Similarly, using \eqref{E53}, \eqref{E54}, and \eqref{E55}, we obtain 
\begin{equation}
\label{E510}
\begin{split}
III
&= (K_2+K_3) \theta^{-1} A(\rho^\alpha)^{1/3} E(\rho^\alpha)^{2/3} \\
&\le (K_2+K_3) \theta^{-1} \Big(\frac{\rho}{\rho^\alpha}\Big)^{1/3} A(\rho)^{1/3} \Big(\frac{\rho}{\rho^\alpha}\Big)^{2/3} E(\rho)^{2/3} \\
&\le (K_2+K_3) \rho^{-\beta} \rho^{1-\alpha} (K_1 \cE)^{1/3} (4 \rho^{1/2} \cE)^{2/3} \\
&\le K_1^{1/3} (K_2+K_3) \cE.
\end{split}
\end{equation}
\item
Finally, we set $K_4 = C K_3 + 2 K_3 + K_1^{1/3} (K_2+K_3)$ and 
\[
\cE = \frac{1}{2} \min\set{1, (\varepsilon/K_4)^2}
\]
where $\varepsilon$ is the absolute number in Lemma \ref{L1}.
Then from \eqref{E57}, \eqref{E58}, \eqref{E59}, and \eqref{E510}, we conclude that 
\[
Y(z,\theta \rho^{\alpha}) 
\le C K_3 \cE^{2/3} + 2 K_3 \cE^{1/2} + K_1^{1/3} (K_2+K_3) \cE 
\le K_4 \cE^{1/2} < \varepsilon.
\]
By Lemma \ref{L1} $\bv$ is regular at a boundary point $z \in \Gamma_T$ and this completes the proof of Proposition \ref{P}. 
\end{enumerate}
\end{proof}

\begin{proof}[Proof of Theorem \ref{T}]
We may consider the set $\cS$ of boundary singular points in the unit cylinder $Q$.
Proposition \ref{P} implies that if $z$ is a boundary singular point, 
then for all $r < 2^{-12}$
\[
\cE r^{3/2}
\le \int_{Q^+(r)} |\bv|^{10/3} + |\nabla \bv|^2 + |\pi-\inp{\pi}|^{5/3} + |\nabla \pi|^{5/4} dxdt.
\]
Fix $5r < 2^{-12}$ and consider the covering $\set{Q(r) : z \in \cS}$.
By the Vitali covering lemma, there is a finite disjoint sub-family
\[
\set{Q(z_j,r) : j=1,2,\dots,M}
\]
such that $\cS \subset \bigcup Q(z_j,5r)$.
Summing the inequality above at $z_j$ for $j=1,2,\dots,M$ yields
\begin{align*}
M \cE r^{3/2}
&\le \sum_{i=1}^{M} \int_{Q^+(z_j,r)} |\bv|^{10/3} + |\nabla \bv|^2 + |\pi-\inp{\pi}|^{5/3} + |\nabla \pi|^{5/4} dxdt \\
&\le \int_{Q^+} |\bv|^{10/3} + |\nabla \bv|^2 + |\pi-\inp{\pi}|^{5/3} + |\nabla \pi|^{5/4} dxdt =: K_5 < \infty.
\end{align*}
We denote by $N(r)$ the minimum number of parabolic cylinders $Q(r)$ required to cover the set $\cS$.
Since $N(r) \le M \le K_5 \cE^{-1} r^{-3/2}$, we conclude that 
\[
\limsup_{r\to0} \frac{\log N(r)}{-\log r} \le \frac{3}{2}.
\]
This completes the proof of Theorem \ref{T}.
\end{proof}

\section*{Acknowledgment}

The authors would like to express their sincere gratitude to Dr. Yanqing Wang for pointing out a rough estimate in Lemma 3 in the first draft. 
Due to his comment, the result of this manuscript is considerable improved.
Hi Jun Choe has been supported by the National Research Foundation of Korea (NRF) grant funded by the Korea government(MSIP) (No. 2015R1A5A1009350).
Minsuk Yang has been supported by the National Research Foundation of Korea (NRF) grant funded by the Korea government(MSIP) (No. 2016R1C1B2015731) and (No. 2015R1A5A1009350).

%\section*{References}


\begin{thebibliography}{10}

\bibitem{CKN}
L. Caffarelli, R. V. Kohn, L. Nirenberg,
{\em Partial regularity of suitable weak solutions of the {N}avier-{Sch}tokes equations}, Comm. Pure Appl. Math. {\bf 35} (1982) 771--831.

\bibitem{CL}
H. J. Choe, J. L. Lewis, 
{\em On the singular set in the {N}avier-{Sch}tokes equations}, 
J. Funct. Anal. 175 (2000) 348--369.

\bibitem{ESS}
L. Escauriaza, G. A. Seregin, V. Sverak, 
{\em On backward uniquness for parabolic equations}, 
Arch. Rational Mech. Anal., 169 (2003), 145-157.

\bibitem{F}
K. Falconer, 
Fractal geometry, 3rd Edition, John Wiley \& Sons, Ltd., Chichester, 2014, mathematical foundations and applications.

\bibitem{GS}
Y. Giga, H. Sohr, 
{\em Abstract $L^p$ estimates for the Cauchy problem with applications to the Navier?Stokes equations in exterior domains}, 
J. Funct. Anal. 102 (1991) 72-94.

\bibitem{GKT}
S. Gustafson, K. Kang, T. P. Tsai,
{\em Regularity criteria for suitable weak solutions of the Navier--Stokes equations near the boundary},
J. Differ. Equ. 226 (2006) 594-618.

\bibitem{H}
E. Hopf, 
{\em Uber die Anfangswertaufgabe fur die hydrodynamischen Grundgleichungen}. 
Math. Nachr. 4, (1950), 213-231.

\bibitem{Kang}
K. Kang, 
{\em Unbounded normal derivative for the Stokes system near boundary}, 
Mathemtische Annalen 331 (2005), 87-109.

\bibitem{Kato}
T. Kato, 
{\em Strong Lp-solutions of the Navier-Stokes equations in Rn with applications to weak solutions}, 
Math. Z., 197 (1984), 471-480.

\bibitem{KY}
Y. Koh, M. Yang,
{\em The Minkowski dimension of interior singular points in the incompressible Navier--Stokes equations},
J. Differ. Equ. 261 (2016) 3137-3148.

\bibitem{Ku}
I. Kukavica, 
{\em The fractal dimension of the singular set for solutions of the Navier-Stokes system},
Nonlinearity 22 (2009) 2889--2900.

\bibitem{KP}
I. Kukavica, Y. Pei, 
{\em An estimate on the parabolic fractal dimension of the singular set for solutions of the {N}avier-{Sch}tokes system}, 
Nonlinearity 25 (2012) 2775--2783.

\bibitem{LS}
O. A. Ladyzhenskaya, G. A. Seregin,
{\em On partial regularity of suitable weak solutions to the three-dimensional {N}avier-{Sch}tokes equations}, 
J. Math. Fluid Mech. 1 (1999) 356--387.

\bibitem{Leray}
J. Leray, 
{\em Sur le mouvement d\'eun liquide visqueux emplissant l\'espace}, 
Acta Math., 63 (1934), 193-248.

\bibitem{Lin}
F. Lin,
{\em A new proof of the {C}affarelli-{K}ohn-{N}irenberg theorem}, 
Comm. Pure Appl. Math. 51~(3) (1998) 241--257.

\bibitem{RS}
J. C. Robinson, W. Sadowski,
{\em On the dimension of the singular set of solutions to the {N}avier-{Sch}tokes equations}, 
Comm. Math. Phys. 309~(2) (2012) 497--506.

\bibitem{Sch}
V. Scheffer, 
{\em Hausdorff measure and the {N}avier-{Sch}tokes equations}, 
Comm. Math. Phys. 55~(2) (1977) 97--112.

\bibitem{Se1}
G. A. Seregin, 
{\em Some estimates near the boundary for solutions to the non-stationary linearized Navier--Stokes equations}, 
Zap. Nauchn. Sem. S.-Peterburg. Otdel. Mat. Inst. Steklov. (POMI) 271 (2000) 204-223.

\bibitem{Se2}
G. A. Seregin, 
{\em Local regularity of suitable weak solutions to the Navier-Stokes equations near the boundary}, 
Journal of Mathematical Fluid Mechanics 4 (2002) 1-29.

\bibitem{Se3}
G. A. Seregin, 
{\em A note on local boundary regularity for the Stokes system}, 
Zap. Nauchn. Semin. POMI 370 (2009), 151-159.

\bibitem{SS}
G. A. Seregin, T. N. Shilkin,
{\em The local regularity theory for the Navier--Stokes equations near the boundary},
arXiv:1402.7181.

\bibitem{SSS}
G. A. Seregin, T. N. Shilkin, V. A. Solonnikov,
{\em Partial boundary regularity for the Navier--Stokes equations},
Journal of Mathematical Sciences, 132 (2006) 339-358.

\bibitem{WW}
Y. Wang, G. Wu,
{\em On the box-counting dimension of the potential singular set for suitable weak solutions to the 3D Navier--Stokes equations},
Nonlinearity 30 (2017) 1762-1772.

\end{thebibliography}
\end{document}